\documentclass[12pt]{amsart}

\usepackage{hyperref}

\newtheorem{thm}{Theorem}
\newtheorem{lem}[thm]{Lemma}
\newtheorem{crl}[thm]{Corollary}
\newtheorem{rmk}{Remark}

\numberwithin{equation}{section}
\numberwithin{thm}{section}
\numberwithin{rmk}{section}
\numberwithin{nt}{section}
\numberwithin{Example}{section}

\newcommand{\f}[1]{{\boldsymbol{#1}}}

\DeclareMathOperator{\Ker}{{{Ker}}}
\DeclareMathOperator{\im}{{{Im}}}

\DeclareMathOperator{\Gis}{{\mathtt{Gen}}}

\newcommand{\alp}{\alpha}
\newcommand{\gam}{\gamma}

\newcommand{\Lam}{\Lambda}
\newcommand{\ome}{\omega}
\newcommand{\Ome}{\Omega}

\newcommand{\ten}{\otimes}

\newcommand{\wed}{\wedge}

\newcommand{\Fla}{^{\flat}{}}
\newcommand{\Sha}{^{\sharp}{}}

\newcommand{\db}[1]{\,{{#1}\!{#1}}\,}

\newcommand{\M}[1]{{\mathcal{#1}}}

\newcommand{\com}{{}\circ{}}

\newcommand{\tecka}{\text{\Large{.}}}

\begin{document}

\title[On Lie algebras of generators]{On Lie algebras of generators of infinitesimal symmetries of almost-cosymplectic-contact structures }

\author[J. Jany\v ska]{Josef Jany\v ska}

\thanks{Supported
by the  grant GA \v CR 14--02476S}

\address{\ 
\newline
Department of Mathematics and Statistics, Masaryk University
\newline
Kotl\'a\v rsk\'a 2, 611 37 Brno, Czech Republic
\newline
e-mail: janyska@math.muni.cz}
  
\begin{abstract}
We study Lie algebras of generators of infinitesimal symmetries of almost-cosymplectic-contact structures of  odd dimensional manifolds.
The almost-cosymplectic-contact structure admits on the sheaf of pairs  of 1-forms and functions the structure of a Lie algebra. 
We describe Lie subalgebras in this Lie algebra given by pairs generating infinitesimal symmetries of basic tensor fields given by the  almost-cosymplectic-contact structure. 
\end{abstract}

\maketitle

\bigskip

\noindent{{\it Keywords}:
Almost-cosymplectic-contact structure; almost-coPoisson-Jacobi structure; infinitesimal symmetry; Lie algebra.
\\
{\ }
}

\section*{Introduction}
\setcounter{equation}{0}

The (7-dimensional) phase space of the (4-dimensional) classical spa\-ce\-time can be defined as the space of 1-jets of motions, \cite{JanMod08}. A Lorentzian metric and an electromagnetic field then define on the phase space the geometrical structure given by a 1-form $ \omega $ and a 2-form $ \Omega $ such that $ \omega \wed \Omega^{3} \not\equiv 0 $ and $ d\Omega = 0 $.
In \cite{JanMod09} such structure was generalized for any odd-dimensional manifold $ \f M $ under the name almost-cosymplectic-contact structure.
The almost-cosymplectic-contact structure on $ \f M $ admits a Lie brac\-ket $ \db[ , \db] $ of pairs $ (\alpha,h)  $ of 1-forms and functions which define a Lie algebra structure on the sheaf
$ \Omega^{1}(\f M) \times C^{\infty}(\f M) $.

In \cite{Jan15, JanVit12} we have studied infinitesimal symmetries of the   almost-cosymplectic-contact structure of the classical phase space. In this paper we shall study infinitesimal symmetries of basic fields generating almost-cosymplectic-contact structure on any odd dimensional manifold. We shall prove that such infinitesimal symmetries are generated by pairs $ (\alpha,h)  $ satisfying certain properties and the restriction of $ \db[ , \db] $
to the subsheaf of generators of infinitesimal symmetries defines Lie subalgebras in $ (\Omega^{1}(\f M) \times C^{\infty}(\f M); \db[ , \db]) $.

\smallskip
In the paper all manifolds and mappings are assumed to be smooth.

\section{Preliminaries}
\setcounter{equation}{0}

We recall some basic notions used in the paper.

\smallskip
{\bf Schouten-Nijenhuis bracket.}
Let us denote by $\mathcal{V}^{p}(\f M) $ the sheaf of skew symmetric
contravariant tensor fields of type $(p,0)$. As the {\em Schouten-Nijenhuis
bracket} (see, for instance, \cite{Vai94}) we assume the 1st order bilinear natural differential operator (see \cite{KolMicSlo93})
$$
[,] : \mathcal{V}^{p}(\f M) \times  \mathcal{V}^{q}(\f M) \to
\mathcal{V}^{p+q-1}(\f M)
$$
given by
\begin{equation}\label{Eq: 1.1a}
i_{[P,Q]} \beta = (-1)^{q(p+1)} i_P d i_Q \beta
+ (-1)^{p} i_Q d i_P \beta
-  i_{P \wedge Q} d\beta 
\end{equation}
for any $P\in \mathcal{V}^{p}(\f M)$, $Q\in \mathcal{V}^{q}(\f M)$ and $(p+q-1)$-form $\beta$. Especially, for a vector field $X$, we have
$
[X,P] = L_X P\,.
$
The Schouten-Nijenhuis bracket is a generalization of the Lie bracket
of vector fields.

We have the following identities
\begin{align}
[P,Q]
& =
(-1)^{pq} \, [Q,P]\,,
\\
[P,Q\wedge R]
& =
[P,Q] \wedge R 
+ (-1)^{pq + q} \, Q\wedge [P,R]\,,
\end{align}
where $R\in \mathcal{V}^{r}(\f M)$. Further we have the (graded) Jacobi identity
\begin{align}
(-1)^{p(r-1)} \, \big[P,[Q,R]\big] 
+ (-1)^{q(p-1)} \, \big[Q,[R,P]\big]
&
\\ 
+ (-1)^{r(q-1)} \, \big[R,[P,Q]\big]
& = \nonumber
0\,. 
\end{align}

\smallskip
{\bf Structures of odd dimensional manifolds.}
Let $\f M$ be a $(2n+1)$-dimensional manifold.

A
\emph{pre cosymplectic (regular) structure (pair)} 
on $\f M$ is given by a 1-form $\ome$ and a 2-form $\Ome$
such that
$
\ome\wed \Ome^n \not\equiv 0\,.
$
A \emph{contravariant (regular) structure (pair)} $(E,\Lam)$ is given by
a vector field $E$ and a skew symmetric 2-vector field $\Lam$ such that
$
E\wed \Lam^n \not\equiv 0\,.
$
We denote by
$
\Ome\Fla:T\f M\to T^*\f M\,
$
and
$
\Lam\Sha:T^*\f M\to T\f M\,
$
the corresponding "musical" morphisms.

By \cite{Lich78}
if $(\ome,\Ome)$ is a pre cosymplectic pair then
there exists a unique regular pair
$(E,\Lam)$ such that
\begin{equation}\label{Eq: 1.1}
(\Ome\Fla_{|\im(\Lam\Sha)})^{-1}
= \Lam\Sha_{|\im(\Ome\Fla)} \,,
\quad i_E\ome =1\,,
\quad i_E\Ome =0\,,
\quad i_\ome\Lam = 0\,.
\end{equation}
On the other hand for any regular pair $(E,\Lam)$
there exists a unique (regular) pair $(\ome,\Ome)$
satisfying the above identities.
The pairs $(\ome,\Ome)$ and
$(E,\Lam)$ satisfying the above identities are said to be mutually 
\emph{dual}.
The vector field $E$ is usually called
the 
\emph{Reeb vector field}
of the pair $(\ome,\Ome)$. In fact
geometrical structures given by dual pairs coincide. 

An 
\emph{almost-cosymplectic-contact (regular) structure (pair)}
\cite{JanMod09}
is
given by
a pair 
$(\ome,\Ome)$
such that
\begin{equation}\label{Eq: 1.2}
d\Ome = 0 \,,\qquad
\ome \wed \Ome^n \not\equiv 0 \,.
\end{equation}
The dual 
\emph{almost-coPoisson-Jacobi structure (pair)\/} 
is given by the
pair $(E,\Lam)$
such that
\begin{equation}\label{Eq: 1.3}
[E,\Lam] = - E\wed \Lam\Sha(L_E\ome)\,,
\qquad
[\Lam,\Lam] = 2 \, E\wed (\Lam\Sha\ten\Lam\Sha)(d\ome)\,.
\end{equation}
Here $[,]$ is the Schouten-Nijenhuis bracket \eqref{Eq: 1.1a}.

\begin{rmk}\label{Rm: 1.1}
{\rm
An almost-cosymplectic-contact pair generalizes standard cosymplectic and contact pairs.
Really, if $d\ome =0$ we obtain a cosymplectic pair (se, for instance, \cite{deLTuy96}).
The dual \emph{coPoisson pair} (see \cite{JanMod09}) is
given by the pair
$(E, \Lam)$
such that
$
[E, \Lam] = 0 \,,
$
$
[\Lam, \Lam] = 0 \,.
$
A 
{contact structure (pair)}
is given by a pair 
$(\ome,\Ome)$
such that
$
\Ome = d\ome \,,
$
$
\ome \wed \Ome^n \not\equiv 0 \,.
$
The dual
{\it Jacobi structure (pair)}
is given by the pair
$(E, \Lam)$
such that
$
[E, \Lam] = 0 \,,
$
$
[\Lam, \Lam] = - 2 E \wed \Lam \,
$ (see \cite{Kir76}).
}
\end{rmk}

\begin{rmk}\label{Rm: 1.2}
{\rm
Given an almost-cosymplectic-contact regular pair $ (\omega,\Omega) $ we can consider the second pair  $ (\omega,F=\Omega + d\omega) $ which is 
almost-cosymplectic-contact but generally  need not be regular.
}
\end{rmk}

\smallskip
{\bf Splitting of the tangent bundle.}
In what follows we assume an odd dimensional manifold $ \f M $ with a regular almost-cosymplectic-contact structure $ (\omega,\Omega)  $. We assume the dual (regular) almost-coPois\-son-Jacobi structure  $(E,\Lambda)$. 
 Then we have $ \Ker (\omega) = \im(\Lambda\Sha) $ and $ \Ker (E) = \im(\Omega\Fla) $ and we have the splitting
$$
T\f M =  \im(\Lambda\Sha)\oplus \langle E \rangle\,, \qquad
T^*\f M = \im(\Omega\Fla) \oplus \langle \omega \rangle\,,
$$ 
i.e. any vector field $X$ and any 1-form $\beta$ can be decomposed as
\begin{equation}\label{Eq: 1.4}
X = X_{(\alpha,h)} = \alpha\Sha + h\, E\,, \qquad
\beta = \beta_{(Y,f)} = Y\Fla + f\, \omega\,,
\end{equation}
where $h,f\in C^\infty(\f M)$, $\alpha$ be a 1-form and $Y$ be a vector field. 
In what follows we shall use notation $\alpha\Sha = \Lambda\Sha(\alpha)$ and $Y\Fla = \Omega\Fla(Y)$.
Moreover, $h = \omega(X_{(\alpha,h)})$ and $f = \beta_{(Y,f)}(E)$. Let us note that the splitting \eqref{Eq: 1.4} is not defined uniquely, really $ X_{(\alpha_1,h_1)} = X_{(\alpha_2,h_2)}$ if and only if $\alpha_1\Sha = \alpha_2\Sha$ and $h_1= h_2$, i.e. $\alpha_1\Sha - \alpha_2\Sha = 0$ that means that $\alp_1 - \alpha_2 \in \langle\omega\rangle$. Similarly
$\beta_{(Y_1,f_1)} = \beta_{(Y_2,f_2)}$ if and only if $Y_1 - Y_2 \in \langle E \rangle$ and $f_1 = f_2$.

The projections $p_2:T\f M \to \langle E \rangle$ and $p_1: T\f M \to \im(\Lambda\Sha) = \Ker(\omega)$ are given by $ X \mapsto \omega(X) \, E $ and $ X \mapsto X - \omega(X) \, E $\,. Equivalently, the projections $q_2:T^{*}\f M \to \langle \omega \rangle$ and $q_1: T^{*}\f M \to \im(\Omega\Fla) = \Ker(E)$ are given by $ \beta \mapsto \beta(E) \, \omega $ and $ \beta \mapsto \beta - \beta(E) \, \omega $\,.
Moreover, $\Lambda\Sha \com \Omega\Fla = p_1$ and $\Omega\Fla \com \Lambda\Sha = q_1$.

\section{Lie algebras of generators of infinitesimal symmetries}
\setcounter{equation}{0}
We shall study infinitesimal symmetries of basic tensor
fields generating the almost-cosymplectic-contact and the dual 
almost-coPoisson-Jacobi structures. 

\subsection{Lie algebra of pairs of 1-forms and functions}
The almost-cosymplectic-contact structure allows us to define a Lie algebra structure on the sheaf $\Omega^1(\f M)\times  C^\infty(\f M)$ of 1-forms and functions. 

\begin{lem}\label{Lm: 2.1}
Let us assume two vector fields $ X_{(\alpha_i,h_i)} = \alpha_i\Sha + h_i\, E $, $i=1,2$, on $\f M$. Then
\begin{align}\label{Eq: 2.1}
[X_{(\alpha_1  ,h_1)} & ,X_{(\alpha_2,h_2)}] 
 =
\big(
d\Lambda(\alpha_1,\alpha_2) 
- i_{\alpha_2\Sha} d\alpha_1 
+ i_{\alpha_1\Sha} d\alpha_2
\\
& \nonumber
- \alpha_1(E) \, (i_{\alpha_2\Sha}d\omega)
+ \alpha_2(E) \, (i_{\alpha_1\Sha}d\omega)
\\
& \nonumber
+ h_1 \,( L_E\alpha_2 -  \alpha_2(E) \, L_E \omega) - h_2 \,( L_E\alpha_1 -  \alpha_1(E) \, L_E \omega)
\big)\Sha 
\\
& \nonumber
+ \big(
\alpha_1\Sha{\tecka} h_2 -  \alpha_2\Sha{\tecka}h_1 - d\omega(\alpha_1\Sha, \alpha_2\Sha)
\\
& \nonumber
+ h_1 \,(E{\tecka}h_2 + \Lambda(L_E\omega, \alpha_2)) 
\\
& \nonumber
-  h_2 \, ( E{\tecka}h_1 +  \Lambda(L_E\omega, \alpha_1))
\big) \, E\,.
\end{align}
\end{lem}

\begin{proof}
It follows from (see \cite{JanMod09})
\begin{align}\label{Eq: 2.2}
	[E,\alpha\Sha] 
& =
	\big(L_E \alpha - \alpha(E)\, (L_E\omega)\big)\Sha + 
\Lambda(L_E\omega,\alpha) \, E \,,
\\
	[\alpha\Sha, \beta\Sha] \label{Eq: 2.3}
&=
	\big(d\Lambda(\alpha,\beta) 
- i_{\beta\Sha} d\alpha 
+ \alpha(E) \, (i_{\beta\Sha}d\omega)
\\
& \quad\nonumber
+ i_{\alpha\Sha} d\beta
- \beta(E) \, (i_{\alpha\Sha}d\omega)
\big) \Sha 
- d\omega(\alpha\Sha, \beta\Sha) \, E \,.
\end{align}
Then
\begin{align*}
	[X_{(\alpha_1  ,h_1)}  ,X_{(\alpha_2,h_2)}] 
& =
[\alpha_1\Sha, \alpha_2\Sha] + h_2 [\alp_1\Sha,E]  + h_1 [E,\alpha_2\Sha] 
\\ 
& \quad 
+ \big(\alpha_1\Sha{\tecka}h_2 - \alpha_2\Sha{\tecka}h_1 + h_1 E{\tecka}h_2 - h_2 E{\tecka}h_1
\big)\, E
\end{align*}
and from \eqref{Eq: 2.2} and \eqref{Eq: 2.3} we get Lemma \ref{Lm: 2.1}.
\end{proof} 

As a consequence of Lemma \ref{Lm: 2.1} we get the Lie bracket of pairs $(\alpha_i,h_i)\in \Omega^1(\f M)\times  C^\infty(\f M)$ given by
\begin{align}\label{Eq: 2.4}
\db[(\alpha_1  ,h_1) & ;(\alpha_2,h_2)\db] 
 =
\big(
d\Lambda(\alpha_1,\alpha_2) 
- i_{\alpha_2\Sha} d\alpha_1 
+ i_{\alpha_1\Sha} d\alpha_2
\\
& \nonumber
+ \alpha_1(E) \, (i_{\alpha_2\Sha}d\omega)
- \alpha_2(E) \, (i_{\alpha_1\Sha}d\omega)
\\
& \nonumber
+ h_1 \, (L_E\alpha_2 -  \alpha_2(E) \, L_E \omega)
- h_2 \, (L_E\alpha_1
- \alpha_1(E) \, L_E \omega)
\, ;
\\
& \nonumber
\alpha_1\Sha{\tecka}h_2 -  \alpha_2\Sha{\tecka}h_1 - d\omega(\alpha_1\Sha, \alpha_2\Sha)
\\
& \nonumber
+ h_1 \,(E{\tecka}h_2 + \Lambda(L_E\omega, \alpha_2)) 
-  h_2 \, ( E{\tecka}h_1 +  \Lambda(L_E\omega, \alpha_1))
\big) 
\end{align}
which defines a Lie algebra structure on 
$\Omega^1(\f M)\times  C^\infty(\f M)$ given by the almost-cosymplectic-contact structure $(\omega,\Omega)$. Moreover, we have
\begin{equation*}
X_{\db[(\alpha_1  ,h_1)  ;(\alpha_2,h_2)\db]}
=
[X_{(\alpha_1  ,h_1)}  ,X_{(\alpha_2,h_2)}] \,.
\end{equation*}

\smallskip
Let $ T $ be a tensor field of any type.  An \emph{infinitesimal symmetry}  of $T$ is a vector
field $ X $ on $\f M $ such that $L_X  T = 0$. From 
$$ 
L_{[X,Y]} = L_X L_Y - L_YL_X 
$$ 
it follows that infinitesimal symmetries of $T$ form a Lie subalgebra, denoted by $\mathcal{L}(T)$, of the Lie algebra $(\M V^1(\f M); [,])$ of vector fields on $\f M$. Moreover, the Lie subalgebra $(\mathcal{L}(T); [,])$ is generated by the Lie subalgebra $(\Gis(T); \db[,\db]) \subset (\Omega^1(\f M)\times  C^\infty(\f M); \db[ , \db])$ of generators of infinitesimal symmetries of $T$. 

\begin{rmk}\label{Rm: 2.1}
{\rm
 Let as recall that a {\it Lie algebroid structure} of a a vector
bundle $\pi:\f E \to \f M$ is defined by
 (see, for instance, \cite{Mac87}):

–- a composition law $(s_1, s_2) \longmapsto \db[s_1, s_2\db]$ on the space $\Gamma(\pi)$ of smooth sections of $\f E$, for which $\Gamma(\pi)$ becomes a Lie algebra,

–- a smooth vector bundle map $\rho : \f E \to T\f M$, where $T\f M$ is the tangent bundle of $\f M$, which satisfies the following two properties:

(i) the map $s \to  \rho \com s$ is a Lie algebras homomorphism from the Lie algebra $(\Gamma(\pi); \db[,\db])$ into the Lie algebra $(\M V^{1}(\f M);[,])$;

(ii) for every pair $(s_1, s_2)$ of smooth sections of $\pi$, and every smooth function $f : \f M \to \Bbb R$, we have the Leibniz-type formula,
\begin{equation}\label{Eq: 1.5}
\db[ s_1, f\, s_2\db] = f\,\db[s_1, s_2\db] +
(i_{(\rho \com s_1)}df)\,
s_2\,.
\end{equation}

The vector bundle $\pi:\f E \to \f M$ equipped with its Lie algebroid structure will
be called a {\it Lie algebroid}; the composition law
$(s_1, s_2) \mapsto \db[ s_1, s_2\db]$ will be called the {\it bracket} and the map $\rho : \f E \to T\f M$ the
{\it anchor}.

The pair $(\alpha,h)$ can be considered as a section $\f M \to T^*\f M \times \mathbb{R}$ and the bracket \eqref{Eq: 2.4} defines the Lie bracket of sections of the vector bundle $\f E = T^*\f M \times  \mathbb{R} \to \f M$. A natural question arise if this bracket defines on $\f E$ the structure of a Lie algebroid with the anchor 
$\rho : \f E \to T\f M$ such that
 $\rho \circ (\alpha,h) = X_{(\alpha,h)} $. The answer is no because, for $ f\in C^{\infty}(\f M) $,
\begin{align*}
\db[(\alpha_1  ,h_1) ;f\,(\alpha_2,h_2)\db] 
& =
 f\, \db[(\alpha_1  ,h_1)  ;(\alpha_2,h_2)\db] 
 \\
 & \quad
 + (X_{(\alpha_1,h_1)}{\tecka}f)\, (\alpha_2,h_2) + \Lambda(\alpha_1,\alpha_2)\, df\,,
\end{align*}
i.e., the Leibniz-type formula \eqref{Eq: 1.5} is not satisfied. 
}
\end{rmk}

\subsection{Infinitesimal symmetries of $ \omega $}

\begin{thm}\label{Th: 2.2}
A vector field $ X $ on $ \f M $ is an infinitesimal symmetry of $ \omega $, i.e. $ L_X\omega = 0 $, if and only if $ X = X_{(\alpha,h)} $, where
$\alpha$ and $h$ are related by the following condition
\begin{equation}\label{Eq: 2.5}
i_{\alpha\Sha}d\omega + h\, i_E\, d\omega + dh = 0\,.
\end{equation}
\end{thm}

\begin{proof}
Any vector field on $\f M$ is of the form $ X_{(\alpha,h)} $.
Then we get
$$
0 = L_{X_{(\alpha,h)}}\omega = i_{\alpha\Sha}d\omega + i_{h\, E} d\omega + di_{\alpha\Sha}\omega + di_{h\, E}\omega
$$
and from $i_{\alpha\Sha}\omega=0$ and $i_E \omega = 1$ Theorem \ref{Th: 2.2} follows.
\end{proof}

\begin{lem}\label{Lm: 2.3}
A vector field $X_{(\alpha,h)}$ is an infinitesimal symmetry of $\omega$
if and only if the following equations are satisfied:
\begin{enumerate} 
\item  $i_E dh + i_E i_{\alpha\Sha}d\omega
= E{\tecka}h + \Lambda(L_E\ome,\alpha) = 0$\,,
\item $ d\omega(\alpha\Sha,\beta\Sha) + h \, d\omega(E,\beta\Sha) + dh(\beta\Sha) = 0$ for any 1-form $ \beta $\,.
\end{enumerate}
\end{lem}

\begin{proof}
If we evaluate the 1-form on the left hand side of
  \eqref{Eq: 2.5} on the Reeb vector field $E$ we get 
$ i_E dh + i_E i_{\alpha\Sha}d\omega
 = E{\tecka} h - i_{\alpha\Sha}i_E d\omega 
 = E{\tecka}h - \Lam(\alpha,L_E\omega)
= 0 $.  
On the other hand if we evaluate this form on $\beta\Sha$, for any 1-form $\beta$, we get (2).

The inverse follows from the splitting $ T\f M =  \im(\Lambda\Sha)\oplus \langle E \rangle $, i.e. a 1-form with zero values on $E$ and $\beta\Sha$, for any 1-form $\beta$, is the zero form.
\end{proof}

\begin{lem}\label{Lm: 2.4}
Let us assume two infinitesimal symmetries $ X_{(\alpha_i,h_i)} = \alpha_i\Sha + h_i\, E $, $i=1,2$, of $\omega$. Then
\begin{align}\label{Eq: 2.6}
[X_{(\alpha_1  ,h_1)} & ,X_{(\alpha_2,h_2)}] 
 =
\big(
d\Lambda(\alpha_1,\alpha_2) 
- i_{\alpha_2\Sha} d\alpha_1 
+ i_{\alpha_1\Sha} d\alpha_2
\\
& \nonumber
+ \alpha_1(E) \, (i_{\alpha_2\Sha}d\omega)
- \alpha_2(E) \, (i_{\alpha_1\Sha}d\omega)
\\
& \nonumber
+ h_1 \, (L_E\alpha_2 -  \alpha_2(E) \, L_E \omega)
- h_2 \, (L_E\alpha_1
- \alpha_1(E) \, L_E \omega)
\big)\Sha 
\\
& \nonumber
+ \big(
\alpha_1\Sha{\tecka}h_2 -  \alpha_2\Sha{\tecka}h_1 - d\omega(\alpha_1\Sha, \alpha_2\Sha)
\big) \, E\,
\end{align}
and we obtain the bracket
\begin{align}\label{Eq: 2.7}
\db[(\alpha_1  ,h_1) & ;(\alpha_2,h_2)\db] 
 =
\big(
d\Lambda(\alpha_1,\alpha_2) 
- i_{\alpha_2\Sha} d\alpha_1 
+ i_{\alpha_1\Sha} d\alpha_2
\\
& \nonumber
+ \alpha_1(E) \, (i_{\alpha_2\Sha}d\omega)
- \alpha_2(E) \, (i_{\alpha_1\Sha}d\omega)
\\
& \nonumber
+ h_1 \, (L_E\alpha_2 -  \alpha_2(E) \, L_E \omega)
- h_2 \, (L_E\alpha_1
- \alpha_1(E) \, L_E \omega)
\, ;
\\
& \nonumber
\alpha_1\Sha{\tecka}h_2 -  \alpha_2\Sha{\tecka}h_1 - d\omega(\alpha_1\Sha, \alpha_2\Sha)
\big) 
\\
& \qquad\qquad\nonumber =
\big(
d\Lambda(\alpha_1,\alpha_2) 
- i_{\alpha_2\Sha} d\alpha_1 
+ i_{\alpha_1\Sha} d\alpha_2
\\
& \nonumber
- \alpha_1(E) \, dh_2
+ \alpha_2(E) \, dh_1
+ h_1 \, L_E\alpha_2 - h_2 \, L_E\alpha_1
\, ;
\\
& \nonumber
d\omega(\alpha_1\Sha, \alpha_2\Sha)
+ h_1 \, d\omega(E,\alpha_2\Sha) - h_2 \,d\omega(E, \alpha_1\Sha)
\big) \,.
\end{align}
\end{lem}

\begin{proof}
It follows from Lemmas \ref{Lm: 2.1} and \ref{Lm: 2.3} and \eqref{Eq: 2.4}.
\end{proof}

According to  Lemma \ref{Lm: 2.4} the Lie algebra $(\mathcal{L}(\omega); [,])$ is generated by the Lie subalgebra of pairs $(\alpha,h)\in (\Gis(\omega); \db[, \db]) \subset (\Omega^1(\f M)\times  C^\infty(\f M); \db[, \db])$ satisfying the condition \eqref{Eq: 2.5}
(or conditions (1) and (2) of Lemma \ref{Lm: 2.3}) with the bracket \eqref{Eq: 2.7}. 

\subsection{Infinitesimal symmetries of $ \Omega $}

\begin{thm}\label{Th: 2.5}
A vector field $ X $ on $ \f M $ is an infinitesimal symmetry of $ \Omega $, i.e. $ L_X\Omega = 0 $, if and only if $ X = X_{(\alpha,h)}$, where
\begin{equation}\label{Eq: 2.8}
d\alpha = 0\,, \quad \alpha(E) = 0\,,
\end{equation}
i.e. $\alpha$ is a closed 1-form in  $\Ker(E)$.
\end{thm}

\begin{proof}
We have the splitting \eqref{Eq: 1.4} and consider a vector field $X_{(\beta,h)}$. Then, from $d\Omega=0$ and $i_E \Omega = 0$, we get
$$
0 = L_{X_{(\beta,h)}}\Omega = di_{\beta\Sha}\Omega =
d(\beta\Sha)\Fla =  d(\beta - \beta(E)\, \omega)
$$
which implies that the closed 1-form $\alpha = \beta - \beta(E) \, \omega$ is
such that $\alpha\Sha = \beta \Sha$. Moreover, $\alpha(E) = \beta (E) - \beta(E) \, \omega(E) = 0$.  
\end{proof}

In what follows we shall denote by $\Ker_{cl}(E)$ the sheaf of closed 1-forms vanishing on $E$. 
From Theorem \ref{Th: 2.5} it follows that the Lie algebra $(\mathcal{L}(\Omega); [,]) $ of infinitesimal symmetries of $ \Omega $ is generated by pairs $ (\alpha,h) $, where $ \alpha = \Ker_{cl}(E)  $. In this case the bracket \eqref{Eq: 2.4} is reduced to the bracket
\begin{align}
\db[(
\alpha_{1},h_{1}) \nonumber 
;
&
(\alpha_{2},h_{2})\db]
: =
\big(
d\Lambda(\alpha_{1},\alpha_{2}); 
\\
& \label{Eq: 2.10a}
\alpha_{1}\Sha .h_{2} - \alpha_{2}\Sha{\tecka}h_{1} 
- d\omega(\alpha_{1}\Sha,\alpha_{2}\Sha)
\\
&
 + h_1\, (E{\tecka}h_2 + \Lambda(L_E\omega,\alpha_2 ))
\nonumber
- h_2\, (E{\tecka}h_1 + \Lambda(L_E\omega,\alpha_1 ))
\big)\,
\end{align}
which defines a Lie algebra structure on $\Ker_{cl}(E) \times  C^\infty(\f M)$ which can be considered as a Lie subalgebra 
$(\Gis(\Omega); \db[,\db]) \subset (\Omega^1(\f M) \times  C^\infty(\f M); \db[,\db])$. Really, $\Ker_{cl}(E) \times  C^\infty(\f M)$ is closed with respect to the bracket \eqref{Eq: 2.10a} which follows from
\begin{align*}
i_E d \Lambda
&
(\alpha_{1},\alpha_{2})
=
L_E(\Lambda(\alpha_{1},\alpha_{2}))
\\
& =
(L_E\Lam)(\alpha_{1},\alpha_{2}) + \Lambda(L_E\alpha_1,\alpha_2)
+ \Lambda(\alpha_1,L_E\alpha_2)
\\
& =
i_{[E,\Lambda]} (\alpha_1 \wedge \alpha_2)
=
- i_{E \wed (L_E\omega)\Sha} (\alpha_1 \wedge \alpha_2) 
=
0\,. 
\end{align*}

\begin{rmk}
{\rm
Any closed 1-form can be locally considered as $\alpha = df$ for a function $f \in C^\infty(\f M)$. Moreover, from $\alpha\in \Ker_{cl}(E)$, the function $f$ satisfies $df(E) = E{\tecka}f =0$. Hence, infinitesimal symmetries of $\Omega$ are locally generated by pairs of functions $(f,h)$ where $E{\tecka}f = 0$. Lie algebras of local generators of infinitesimal symmetries of the almost-cosymplectic-contact structure are studied in \cite{Jan16}.   
}
\end{rmk}

\subsection{Infinitesimal symmetries of the Reeb vector field}

\begin{thm}\label{Th: 2.6}
A vector field $ X $ on $ \f M $ is an infinitesimal symmetry of $ E $, i.e. $ L_X E = [X,E] = 0 $, if and only if $ X = X_{(\alpha,h)}$, where
$\alpha$ and $h$ satisfy the following conditions
\begin{align}\label{Eq: 2.10}
(L_E\alpha - \alpha(E)\, L_E\omega)\Sha  
& = 
0\,,
\\
E{\tecka}h  + \Lambda(L_E\omega,\alpha) 
& = \label{Eq: 2.11}
0\,.
\end{align}
\end{thm}

\begin{proof}
We have 
$$
0 = [X_{(\alpha,h)}, E] = [\alpha\Sha,E] + [h\, E, E]
$$
and from \eqref{Eq: 2.2} we get
$$
[X_{(\alp,h)}, E] = - \big(
L_E\alpha - \alpha(E)\, L_E\omega
\big)\Sha 
- \big(
E{\tecka}h + \Lambda(L_E\omega,\alpha
\big)\, E\,
$$
which proves Theorem \ref{Th: 2.6}.
\end{proof}

\begin{rmk}
{\rm
The condition \eqref{Eq: 2.10} of Theorem \ref{Th: 2.6} is equivalent to the condition
\begin{equation}
(L_E\alpha)(\beta\Sha) - \alpha(E) \, (L_E\omega)(\beta\Sha) = 0
\end{equation}
for any 1-form $\beta$.
}
\end{rmk}

\begin{lem}\label{Lm: 2.7}
The restriction of the bracket \eqref{Eq: 2.4}
to pairs $(\alpha,h)$ satisfying the conditions \eqref{Eq: 2.10} and \eqref{Eq: 2.11} is the bracket
\begin{align}\label{Eq: 2.12}
\db[(\alpha_1  ,h_1) & ;(\alpha_2,h_2)\db] 
 =
\big(
d\Lambda(\alpha_1,\alpha_2) 
- i_{\alpha_2\Sha} d\alpha_1 
+ i_{\alpha_1\Sha} d\alpha_2
\\
& \quad\nonumber
+ \alpha_1(E) \, (i_{\alpha_2\Sha}d\omega)
- \alpha_2(E) \, (i_{\alpha_1\Sha}d\omega)
\, ;
\\
& \quad \nonumber
 \alpha_1\Sha{\tecka}h_2 -  \alpha_2\Sha{\tecka}h_1 - d\omega(\alpha_1\Sha, \alpha_2\Sha)
\big) 
\\
& = \nonumber
\big(
- d\Lambda(\alpha_1,\alpha_2) 
- L_{\alpha_2\Sha} \alpha_1 
+ L _{\alpha_1\Sha} \alpha_2
\\
& \quad\nonumber
+ \alpha_1(E) \, (L_{\alpha_2\Sha}\omega)
- \alpha_2(E) \, (L_{\alpha_1\Sha}\omega)
\, ;
\\
& \quad\nonumber
\alpha_1\Sha{\tecka}h_2 -  \alpha_2\Sha{\tecka}h_1 - d\omega(\alpha_1\Sha, \alpha_2\Sha)
\big) 
\end{align}
which defines a Lie algebra structure on the subsheaf of $
\Omega^1(\f M)\times C^\infty(\f M)$ of pairs of 1-forms and functions satisfying conditions \eqref{Eq: 2.10} and \eqref{Eq: 2.11}.
\end{lem}

\begin{proof}
It follows from \eqref{Eq: 2.4}, \eqref{Eq: 2.10}  and \eqref{Eq: 2.11}.
\end{proof}

\subsection{Infinitesimal symmetries of $\Lambda$}

\begin{thm}\label{Th: 2.8}
A vector field $ X $ on $ \f M $ is an infinitesimal symmetry of $ \Lambda $, i.e. $ L_X \Lambda = [X,\Lambda] = 0 $, if and only if $ X = X_{(\alpha,h)}$, where
$\alpha$ and $h$ satisfy the following conditions
\begin{equation}\label{Eq: 2.13}
[\alpha\Sha,\Lambda] - E \wedge (dh + h\, L_E\omega)\Sha = 0\,.
\end{equation}
\end{thm}

\begin{proof}
We have
$$
L_{X_{(\alpha,h)}} \Lambda = [\alpha\Sha,\Lambda] + [h\, E,\Lambda]\,. 
$$
Theorem \ref{Th: 2.8} follows from
$$
[h\, E,\Lambda] = h \, [E,\Lambda] - E\wedge dh\Sha = - E\wedge (dh + h\, L_E\omega)\Sha\,.
$$
\vglue-1.3\baselineskip
\end{proof}

\begin{lem}\label{Lm: 2.9}
A vector field $ X_{(\alpha,h)} $ is an infinitesimal symmetry of $ \Lambda $
if and only if the following conditions
\begin{align}\label{Eq: 2.14}
d\omega(\alpha\Sha,\beta\Sha) + h \, d\omega(E,\beta\Sha) + dh(\beta\Sha)  
& = 
0\,,
\\ \label{Eq: 2.15}
\alpha(E)\, d\omega(\beta\Sha,\gamma\Sha) - d\alpha(\beta\Sha,\gamma\Sha)
& =
0\,
\end{align}
are satisfied for any 1-forms $ \beta, \gamma$.
\end{lem}

\begin{proof}
It is sufficient to evaluate the 2-vector field on the left hand side of \eqref{Eq: 2.13} on $\ome, \beta$ and $\beta,\gamma$, where $\beta, \gamma$ are closed 1-forms. We get
$$
i_{[\alpha\Sha,\Lambda] - E \wedge (dh + h\, L_E\omega)\Sha}(\omega \wedge \beta) = - \Lambda(i_{\alpha\Sha}d\omega + h\, L_E\omega + dh,\beta)
$$
which vanishes if and only if \eqref{Eq: 2.14} is satisfied.

On the other hand
\begin{align*}
&
i_{[\alpha\Sha,\Lambda] 
 - 
E \wedge (dh + h\, L_E\omega)\Sha} 
(\beta \wedge \gamma) = 
\\
& \qquad\qquad =
\Lambda(\alpha,d\Lambda(\beta,\gamma)) 
+ \Lambda(\beta,d\Lambda(\gamma,\alpha))
+ \Lambda(\gamma,d\Lambda(\alpha,\beta)) 
\\
& \qquad\qquad\quad
- \beta(E)\,\Lambda(h\, L_E\omega + dh,\gamma)
+ \gamma(E)\,\Lambda(h\, L_E\omega + dh,\beta)
\end{align*}
which, by using \eqref{Eq: 2.14}, can be rewritten as
\begin{align*}
&
i_{[\alpha\Sha,\Lambda] 
 - 
E \wedge (dh + h\, L_E\omega)\Sha} 
(\beta \wedge \gamma) = - \tfrac 12 i_{[\Lambda,\Lambda]} (\alpha\wedge\beta\wedge\gamma)
+
d\alpha(\beta\Sha,\gamma\Sha)
\\
& \qquad\quad
+ \beta(E)\,\Lambda(i_{\alpha\Sha}d\omega,\gamma)
- \gamma(E)\,\Lambda(i_{\alpha\Sha}d\omega,\beta)
\\
& \qquad =
-  i_{E\wedge (\Lambda\Sha\otimes\Lambda\Sha)(d\omega)} (\alpha\wedge\beta\wedge\gamma)
+
d\alpha(\beta\Sha,\gamma\Sha)
\\
& \qquad\quad+ \beta(E)\,\Lambda(i_{\alpha\Sha}d\omega,\gamma)
- \gamma(E)\,\Lambda(i_{\alpha\Sha}d\omega,\beta)
\\
& \qquad =
- \alpha(E)\, d\omega(\beta\Sha,\gamma\Sha) + d\alpha(\beta\Sha,\gamma\Sha)
\end{align*}
which vanishes if and only if \eqref{Eq: 2.15} is satisfied.

On the other hand if \eqref{Eq: 2.14} and \eqref{Eq: 2.15} are satisfied, then the 2-vector field $ [\alpha\Sha,\Lambda] 
 - 
E \wedge (dh + h\, L_E\omega)\Sha $ is the zero 2-vector field.
\end{proof}

\subsection{Infinitesimal symmetries of the almost-cosymplectic-con\-tact structure
and the dual almost-coPoisson-Jacobi structure}

An \emph{infinitesimal symmetry  of the almost-cosymplectic-contact structure} $(\ome,\Ome)$ is a vector
field $ X $ on $\f M $ such that $L_X\omega = 0$ and $L_X\Omega = 0$. On the other hand  an \emph{infinitesimal symmetry  of the almost-coPoisson-Jacobi structure} $(E,\Lambda)$ is a vector
field $ X $ on $\f M $ such that $L_X E = [X,E] = 0$ and $L_X\Omega = [X,\Lambda] = 0$.

\begin{thm}\label{Th: 2.10}
A vector field $ X $ is an infinitesimal symmetry of the almost-cosymplectic-contact structure $(\ome,\Ome)$
if and only if
$
X = X_{(\alpha,h)}\,,
$
where $\alpha \in \Ker_{cl}(E)$ and the condition \eqref{Eq: 2.5} is satisfied.
\end{thm}

\begin{proof}
It follows from Theorems \ref{Th: 2.2} and \ref{Th: 2.5}. 
\end{proof}

\begin{lem}\label{Lm: 2.11}
A vector field $X_{(\alpha,h)}$ is an infinitesimal symmetry of $(\omega,\Omega)$ if and only if
the following conditions are satisfied 
\begin{enumerate}
\item  $\alpha\in\ker_{cl}(\f E)$, i.e. $d\alpha = 0$\,, $\alp(E) = 0$\,, 
\item  $i_E dh + i_E i_{\alpha\Sha}d\omega
= E{\tecka}h + \Lambda(L_E\ome,\alpha) = 0$\,,
\item $ d\omega(\alpha\Sha,\beta\Sha) + h \, d\omega(E,\beta\Sha) + dh(\beta\Sha) = 0$ for any 1-form $ \beta $\,.
\end{enumerate}
\end{lem}

\begin{proof}
It is a consequence of Theorem \ref{Th: 2.10} and Lemma \ref{Lm: 2.3}.
\end{proof}

The bracket \eqref{Eq: 2.4} restricted for generators of infinitesimal symmetries of $(\omega,\Omega)$ gives the bracket
\begin{align}\nonumber
\db[(
&
\alp_{1},h_{1})   ;(\alp_{2},h_{2})\db] =
\\ \label{Eq: 2.16}
& 
 = \big(
d\Lambda(\alpha_{1},\alpha_{2});\alpha_{1}\Sha{\tecka}h_{2} - \alpha_{2}\Sha{\tecka}h_{1} 
- d\omega(\alpha_{1}\Sha,\alpha_{2}\Sha)
\big)
\\ 
& \nonumber
 =
\big(
d\Lambda(\alpha_{1},\alpha_{2});
d\omega(\alpha_{1}\Sha,\alpha_{2}\Sha)
+ h_2 \, \Lambda(L_E\omega,\alpha_{1}) 
- h_1 \, \Lambda(L_E\omega,\alpha_{2}) \big)
\\ 
& \nonumber
 =
\big(
d\Lambda(\alpha_{1},\alpha_{2});
d\omega(\alpha_{1}\Sha,\alpha_{2}\Sha)
+ h_1 \, E\tecka h_{2}
- h_2 \, E\tecka h_{1}
\big)\,
\end{align}
which defines the Lie algebra structure on the subsheaf of $\Ker_{cl}(E)\times C^\infty(\f M)$ given by pairs satisfying the condition \eqref{Eq: 2.5}.  We shall denote the Lie algebra of generators of infinitesimal symmetries of $(\omega,\Omega)$ by
$(\Gis(\omega,\Omega); \db[,\db])$.

\begin{crl}\label{Cr: 2.12}
An infinitesimal symmetry
of the cosymplectic  structure $(\omega,\Omega)$ is a vector field
$
X_{(\alpha,h)}\,,
$
where $ \alpha \in \Ker_{cl}(E) $ and $h$ is a constant. 

Then the bracket \eqref{Eq: 2.4} is reduced to
\begin{equation*}
\db[(\alpha_{1},h_{1});(\alpha_{2},h_{2})\db]
=
\big(
d\Lambda(\alpha_{1},\alpha_{2}); 0
\big)\,.
\end{equation*}
I.e. we obtain the subalgebra $ (\Ker_{cl}(\f E)\times \mathbb{R}, \db[,\db]) $ of generators of infinitesimal symmetries of the cosymplectic structure. 
\end{crl}

\begin{proof}
For the cosymplectic structure we have $d\omega=0$ and \eqref{Eq: 2.5} reduces to
$
dh = 0\,.
$
\end{proof}

\begin{crl}\label{Cr: 2.13}
Any infinitesimal symmetry
of the contact  structure $(\omega,\Omega)$ is
 of local type
\begin{equation}\label{Eq: 2.17}
X_{(dh,-h)} =  dh\Sha - h\,E\,,
\end{equation} 
where $ E{\tecka}h = 0 $. I.e., infinitesimal symmetries of the contact structure are Hamilton-Jacobi lifts of functions satisfying $E{\tecka}h = 0$. 

Then the bracket \eqref{Eq: 2.4} is reduced to
\begin{align*}
\db[(dh_{1},-h_{1});(dh_{2},-h_{2})\db]
& =
\big(
d\{h_{1},h_{2}\}, - \{h_{1},h_{2}\}
\big)
\,.
\end{align*}
I.e., the subalgebra of generators of infinitesimal symmetries of the contact structure is identified with the Lie algebra
$
(C^\infty_{\f E}(\f M),\{,\})\, ,
$
where $ C^\infty_{\f E}(\f M) $ is the sheaf of functions $h$ such that $E{\tecka}h = 0$ and $\{,\}$ is the Poisson bracket. 
\end{crl}

\begin{proof}
For a contact structure we have $d\ome = \Omega$ and \eqref{Eq: 2.5} reduces to
$
i_{\alpha\Sha}  \Omega + dh  = \alpha + dh = 0\,,
$
i.e.
$
\alpha = - dh\,.
$ 
From $\alpha\in \Ker_{cl}(E)$ we get $E{\tecka}h = 0$.
\end{proof}

\begin{rmk}\label{Rm: 2.2}
{\rm
For cosymplectic and contact structures a constant multiple of the Reeb vector field is an infinitesimal symmetry of the structure. It is not true for the almost-cosymplectic-contact structure.
}
\end{rmk}

\begin{lem}\label{Lm: 2.14}
A vector field $X_{(\alpha,h)}$ is an infinitesimal symmetry of $(E,\Lambda)$ if and only if
the following conditions are satisfied 
\begin{enumerate}
\item  $(L_E\alpha)(\beta\Sha) - \alpha(E)\, (L_E\omega)(\beta\Sha)   = 0\,, $
\item  $E{\tecka}h  + \Lambda(L_E\omega,\alpha)  = 0\,,$
\item  $ d\omega(\alpha\Sha,\beta\Sha) + h \, d\omega(E,\beta\Sha) + dh(\beta\Sha)  
 = 
0\,, $
\item $ \alpha(E)\, d\omega(\beta\Sha,\gamma\Sha) - d\alpha(\beta\Sha,\gamma\Sha)
 =
0\, $
\end{enumerate}
 for any 1-forms $ \beta, \gamma$.
\end{lem}

\begin{proof}
From Theorem \ref{Th: 2.6} and Lemma \ref{Lm: 2.9}
$X_{(\alpha,h)}$ is an infinitesimal symmetry of $(E,\Lambda)$ if and only if \eqref{Eq: 2.10}, \eqref{Eq: 2.11}, \eqref{Eq: 2.14} and \eqref{Eq: 2.15} are satisfied.  
\end{proof}

 We shall denote the Lie algebra of generators of infinitesimal symmetries of $(E,\Lambda)$ by
$(\Gis(E,\Lambda); \db[,\db])$.

\begin{rmk}\label{Rm: 2.3}
{\rm
We can describe also the Lie algebras of infinitesimal symmetries of other pairs of basic fields. Especially:

\smallskip
1. The Lie algebra  $(\Gis(E,\Omega); \db[,\db])$
is given by pairs satisfying
\begin{enumerate}
\item  $\alpha\in\ker_{cl}(\f E)$, i.e. $d\alpha = 0$\,, $\alp(E) = 0$\,, 
\item  $E{\tecka}h  + \Lambda(L_E\omega,\alpha)  = 0\,.$
\end{enumerate}

\smallskip
2. The Lie algebra  $(\Gis(\Lambda,\Omega); \db[,\db])$
is given by pairs satisfying
\begin{enumerate}
\item $\alpha\in\ker_{cl}(\f E)$, i.e. $d\alpha = 0$\,, $\alp(E) = 0$\,, 
\item  $ d\omega(\alpha\Sha,\beta\Sha) + h \, d\omega(E,\beta\Sha) + dh(\beta\Sha)  
 = 
0\, $ for any 1-form $\beta$.
\end{enumerate}

\smallskip
3. The Lie algebra  $(\Gis(E,\omega); \db[,\db])$
is given by pairs satisfying
\begin{enumerate}
\item  $(L_E\alpha)(\beta\Sha) - \alpha(E)\, (L_E\omega)(\beta\Sha)   = 0\, $ for any 1-form $\beta$,
\item  $E{\tecka}h  + \Lambda(L_E\omega,\alpha)  = 0\,,$
\item  $ d\omega(\alpha\Sha,\beta\Sha) + h \, d\omega(E,\beta\Sha) + dh(\beta\Sha)  
 = 
0\, $ for any 1-form $\beta$.
\end{enumerate}

\smallskip
4. The Lie algebra  $(\Gis(\Lam,\omega); \db[,\db])$
is given by pairs satisfying
\begin{enumerate}
\item  $E{\tecka}h  + \Lambda(L_E\omega,\alpha)  = 0\,,$
\item  $ d\omega(\alpha\Sha,\beta\Sha) + h \, d\omega(E,\beta\Sha) + dh(\beta\Sha)  
 = 
0\, $ for any 1-form $\beta$,
\item $ \alpha(E)\, d\omega(\beta\Sha,\gamma\Sha) - d\alpha(\beta\Sha,\gamma\Sha)
 =
0\, $ for any 1-forms $\beta, \gamma$.
\end{enumerate}
}
\end{rmk}

\begin{lem}\label{Lm: 2.15}
Let $X$ be a vector field on $\f M$. 
Then
\begin{equation}\label{Eq: 2.18}
L_X\beta\Sha = (L_X\beta)\Sha
\end{equation}
for any 1-form $\beta$
if and only if $X$ is an infinitesimal symmetry of $\Lambda$.
\end{lem}

\begin{proof}
Let  $X = X_{(\alpha,h)}$. Then
\begin{align*}
L_{X_{(\alpha,h)}} \beta\Sha
& =
[\alpha\Sha + h\, E, \beta\Sha]
= \big(
d\Lambda(\alpha,\beta) - i_{\beta\Sha}d\alpha + \alpha(E) i_{\beta\Sha} d\ome
\\
& \quad 
+ i_{\alpha\Sha}d\beta - \beta(E) i_{\alpha\Sha} d\ome
+
h\, L_E \beta - h \, \beta(E)\, L_E \omega
\big)\Sha
\\
& \quad 
- \big(
d\omega(\alpha\Sha,\beta\Sha) + h\, i_{\beta\Sha} L_E\omega
+ i_{\beta\Sha} dh
\big)\, E\,.
\end{align*}
On the other hand we have
\begin{align*}
(L_{X_{(\alpha,h)}} \beta)\Sha
& = \big(
d\Lambda(\alpha,\beta) 
+ i_{\alpha\Sha}d\beta + h\, L_E \beta + \beta(E)\, dh
\big)\Sha\,.
\end{align*}
Then
\begin{gather*}
(L_{X_{(\alpha,h)}}
 \beta)\Sha  - L_{X_{(\alpha,h)}} \beta\Sha
= \big( i_{\beta\Sha}d\alpha - \alpha(E) i_{\beta\Sha} d\ome
\\
+
\beta(E)\, ( dh + h \, L_E \omega 
+ i_{\alpha\Sha} d\ome
)
\big)\Sha
+
\big(
d\omega(\alpha\Sha,\beta\Sha) + h\, i_{\beta\Sha} L_E\omega
+ i_{\beta\Sha} dh
\big)\, E\,.
\end{gather*}
The identity \eqref{Eq: 2.18} is satisfied
if and only if
 \begin{align*}
d\alpha(\beta\Sha, \gamma\Sha) - \alpha(E)\, d\ome(\beta\Sha, \gamma\Sha)
& = 
0\,,
\\
d\omega(\alpha\Sha,\beta\Sha) + h\, d\omega(E,\beta\Sha)
+  dh(\beta\Sha)
& =
0\,
 \end{align*}
 for any 1-form $\gamma$, i.e., by Lemma \ref{Lm: 2.9}, if and only if  $X_{(\alpha,h)}$ is an infinitesimal symmetry of $\Lambda$.
\end{proof}

\begin{thm}\label{Th: 2.16} 
Let $X$ be a vector field on $\f M$. The following conditions
are equivalent:

1. $L_X\omega = 0$ and $L_X\Ome = 0$.

2. $L_X E= [X,E]=0$ and
$L_X\Lambda= [X,\Lambda]=0$.

Hence the Lie algebras 
$(\Gis(\omega,\Omega); \db[,\db])$ and $(\Gis(E,\Lambda); \db[,\db])$ coincides. 
\end{thm}

\begin{proof} 1. $ \Rightarrow $ 2.
If the conditions (1), (2) and (3) in Lemma \ref{Lm: 2.11}  are satisfied then the conditions (1), $\dots$ ,(4) in Lemma \ref{Lm: 2.14} are satisfied.

2. $ \Rightarrow $ 1.
From Lemmas \ref{Lm: 2.3} and  \ref{Lm: 2.14}  it follows that infinitesimal symmetries of $(E,\Lambda)$ are infinitesimal symmetries of $\omega$. 
Now let $ X_{(\alpha,h)} $ be an infinitesimal symmetry of $ (E,\Lambda) $. To prove that $ X_{(\alpha,h)} $ is the infinitesimal symmetry of $ \Omega $ it is sufficient to evaluate $ L_{ X_{(\alpha,h)}} \Omega = d(i_{ X_{(\alpha,h)}} \Omega)$ on pairs of vector fields $ E,\beta\Sha $ and $ \beta\Sha,\gamma\Sha $, where $ \beta,\, \gamma $ are any 1-forms. From \eqref{Eq: 1.3}, \eqref{Eq: 2.2}, \eqref{Eq: 2.3} and $ \Omega(\beta\Sha,\gamma\Sha) = -\Lambda(\beta,\gamma) $ (see \cite{JanMod09}) we get
\begin{align*}
(L_{ X_{(\alpha,h)}} \Omega) & (E,\beta\Sha)
 =
E\tecka (\Omega(\alpha\Sha,\beta\Sha))
- \beta\Sha\tecka (\Omega(\alpha\Sha,E))
- \Omega(\alpha\Sha,[E,\beta\Sha])
\\
& =
- E\tecka (\Lambda(\alpha,\beta))
+ \Lambda(\alpha,L_E\beta)
- \beta(E)\,\Lambda(\alpha,L_E\omega)
\\
& =
- (L_E\Lambda)(\alpha,\beta)
- \Lambda(L_E\alpha,\beta)
- \beta(E)\,\Lambda(\alpha,L_E\omega)
\\
& =
i_{E\wedge(L_E\ome)\Sha} (\alpha\wedge\beta)
- \Lambda(L_E\alpha,\beta)
- \beta(E)\,\Lambda(\alpha,L_E\omega)
\\
& =
- \alpha(E)\, (L_E\omega)(\beta\Sha) + (L_E\alpha)(\beta\Sha)
\end{align*}
which vanishes by (1) of Lemma \ref{Lm: 2.14}.
Similarly
\begin{align*}
(L_{ X_{(\alpha,h)}} & \Omega)  (\beta\Sha,\gamma\Sha)
 =
\beta\Sha\tecka (\Omega(\alpha\Sha,\gamma\Sha))
- \gamma\Sha\tecka (\Omega(\alpha\Sha,\beta\Sha))
- \Omega(\alpha\Sha,[\beta\Sha,\gamma\Sha])
\\
& =
- \beta\Sha\tecka (\Lambda(\alpha,\gamma))
+ \gamma\Sha\tecka (\Lambda(\alpha,\beta))
+ \Lambda(\alpha,d(\Lambda(\beta,\gamma)))
- \Lambda(\alpha,i_{\gamma\Sha}d\beta)
\\
& \quad
+ \beta(E)\, \Lambda(\alpha,i_{\gamma\Sha}d\omega))
+ \Lambda(\alpha,i_{\beta\Sha}d\gamma)
- \gam(E)\, \Lambda(\alpha,i_{\beta\Sha}d\omega))
\\
& =
- \frac 12 i_{[\Lambda,\Lambda]}(\alpha\wedge\beta\wedge\gamma)
\\
& \quad
+ d\alpha(\beta\Sha,\gamma\Sha)
- \gamma(E)\, d\omega(\beta\Sha, \alpha\Sha)
+ \beta(E)\, d\omega(\gamma\Sha, \alpha\Sha)
\\
& =
- i_{E\wedge(\Lambda\Sha\ten\Lambda\Sha)d\omega}(\alpha\wedge\beta\wedge\gamma)
\\
& \quad
+ d\alpha(\beta\Sha,\gamma\Sha)
- \gamma(E)\, d\omega(\beta\Sha, \alpha\Sha)
+ \beta(E)\, d\omega(\gamma\Sha, \alpha\Sha)
\\
& =
 d\alpha(\beta\Sha,\gamma\Sha)
- \alpha(E)\, d\omega(\beta\Sha, \gamma\Sha)
\end{align*} 
which vanishes by (4) of Lemma \ref{Lm: 2.14}. So $ L_{ X_{(\alpha,h)}} \Omega = 0$.
\end{proof}

\subsection{Derivations on the algebra $(\Gis(\omega,\Omega); \db[,\db])$}
\label{Sec: Derivations}

Let us assume the Lie algebra $(\Gis(\Omega); \db[,\db])$ of generators of infinitesimal symmetries of $\Omega$. The bracket $\db[,\db]$ is a 1st order bilinear differential operator
$$
 \Gis(\Omega)  \times  \Gis(\Omega)
\to \Gis(\Omega)\,.
$$

\begin{thm}\label{Th: 2.17}
The 1st order differential operator
$$
D_{(\alpha,h)} : \Gis(\Omega)
\to \Gis(\Omega)
$$
given by
$$
D_{(\alpha_1,h_1)}(\alpha_2,h_2) = \db[(\alpha_1,h_1);(\alpha_2,h_2)\db]
$$
is a derivation on the Lie algebra $( \Gis(\Omega) ,\db[,\db])$.
\end{thm}

\begin{proof}
It follows from the Jacobi identity for $\db[, \db]$.
\end{proof}

We can define a differential operator $L_X: \Omega^1(\f M) \times  C^{\infty}(\f M)
\to \Omega^1(\f M)  \times  C^{\infty}(\f M) $ given by the Lie derivatives with respect to a vector field $X$, i.e.
\begin{equation}\label{Eq: 2.21}
L_X(\alpha,h) = (L_X\alpha,L_X h)\,.
\end{equation}
Generally this operator does not preserve sheaves of generators of infinitesimal symmetries.

\begin{thm}\label{Th: 2.19}
Let $X$ be an infinitesimal symmetry of the almost-cosymplectic-contact structure $(\omega,\Omega)$.
Then the operator $L_X$ is a derivation on the Lie algebra   
$( \Gis(\omega,\Ome) ;\db[,\db])$ of generators of infinitesimal symmetries of
$(\omega,\Omega)$.
\end{thm}

\begin{proof}
First, let us recall that infinitesimal symmetries of $(\omega,\Omega)$ are infinitesimal symmetries of $(E,\Lambda)$. Suppose the bracket \eqref{Eq: 2.16} of generators of infinitesimal symmetries of $(\omega,\Omega)$. 
We have to prove that $L_X$ is an operator on $\Gis(\omega,\Omega)$, 
i.e. that for any $(\alpha,h) \in \Gis(\omega,\Omega)$ the pair 
$(L_X\alpha,L_Xh) \in \Gis(\omega,\Omega)$. 

We have $\alpha\in \Ker_{cl}(E)$, then $L_X\alpha = di_X\alpha$ which is a closed 1-form. Further
$$
L_X(\alpha(E)) = 0 \quad \Leftrightarrow \quad (L_X\alpha)(E) + \alpha(L_XE) = (L_X\alpha)(E)  = 0
$$
and $L_X\alpha\in \Ker_{cl}(E)$.

 Further we have to prove that the pair  $(L_X\alpha,L_Xh)$ satisfies conditions (1) and (2) of Lemma \ref{Lm: 2.3}. From $L_XE =0$ and $L_X\Lambda = 0$ we get $L_XL_E\omega = 0$ and $L_Xd\omega = 0$. Moreover,
$L_X dh = dL_Xh$.
 
 The pair $(\alpha,h)$ satisfies (1) of Lemma \ref{Lm: 2.3} and we get
\begin{align*}
0
& =
L_X\big(dh(E) + \Lambda(L_E\omega,\alp)\big)
\\
& =
d(L_Xh)(E) + \Lambda(L_E\omega,L_X\alpha) = 0
\end{align*}
and the condition (1) of Lemma \ref{Lm: 2.3} for $(L_X\alpha,L_Xh)$ is satisfied.

Similarly, from the condition (2) of Lemma \ref{Lm: 2.3} we have, for any 1-form $\beta$,
\begin{align*}
0
& =
L_X\big(
d\omega(\alpha\Sha,\beta\Sha) + h \, d\omega(E,\beta\Sha) + dh(\beta\Sha)
\big)
\\
& =
\big(
 d\omega(L_X\alpha\Sha,\beta\Sha)
 + (L_X h) \, d\omega(E,\beta\Sha)
 + d(L_Xh)(\beta\Sha)
\big)
\\
& \quad
+ \big(
d\omega(\alpha\Sha,L_X\beta\Sha)
+ h\, d\omega(E,L_X\beta\Sha)
+ h\, d\omega(E,L_X\beta\Sha)
\big)
\,.
\end{align*}
The term in the second bracket is vanishing because of the condition (2) expressed on $L_E\beta\Sha = (L_E\beta)\Sha$.
Hence the condition  (2) of Lemma \ref{Lm: 2.3}  for the pair $(L_X\alpha,L_Xh)$ is satisfied and this pair is in $\Gis(\omega,\Omega)$.

Further, we have to prove
\begin{align*}
L_X\db[(\alpha_{1},h_{1});(\alpha_{2},h_{2})\db]
& = 
\db[(L_X\alpha_{1},L_X h_{1});(\alpha_{2},h_{2})\db] 
\\
& \quad
+ \db[(\alpha_{1},h_{1});(L_X\alpha_{2},L_Xh_{2})\db]\,.
\end{align*}
For the the first parts of the above pairs the  identity 
$$
L_X(d(\Lambda(\alpha_1,\alpha_2))) = d(\Lambda(L_X\alpha_1,\alpha_2))
+ d (\Lambda(\alpha_1, L_X\alpha_2))\,
$$
has to be satisfied.
But 
\begin{align*}
L_X(d & (\Lambda(\alpha_1  ,\alpha_2)))
 =
di_Xdi_\Lambda(\alpha_1 \wedge \alpha_2) 
=
di_{[X,\Lam]}(\alpha_1 \wedge \alpha_2)
+ d i_\Lambda d i_X (\alpha_1 \wedge \alpha_2)
\\
& =
d(([X,\Lambda])(\alpha_1,\alpha_2)) + d(\Lambda(L_X\alpha_1,\alpha_2))
+ d (\Lambda(\alpha_1, L_X\alpha_2))\,
\end{align*}
and for $[X,\Lambda] = L_X \Lam  = 0$ the identity holds.

For the second parts of pairs the following identity has to be satisfied.
\begin{align*}
L_X\big(
&
d\omega(\alpha_{1}\Sha,\alpha_{2}\Sha)
+ h_2 \, \Lambda(L_E\omega,\alpha_{1}) 
- h_1 \, \Lambda(L_E\omega,\alpha_{2})\big)
=
\\
& =
d\omega((L_X\alpha_{1})\Sha,\alpha_{2}\Sha)
+ h_2 \, \Lambda(L_E\omega,L_X\alpha_{1}) 
- (L_Xh_1) \, \Lambda(L_E\omega,\alpha_{2})
\\
& 
+ d\omega(\alpha_{1}\Sha,(L_X\alpha_{2})\Sha)
+ (L_Xh_2) \, \Lambda(L_E\omega,\alpha_{1}) 
- h_1 \, \Lambda(L_E\omega,L_X\alpha_{2})\,.
\end{align*}
If $X$ is the infinitesimal symmetry of $(\omega,\Omega)$ then it is also the infinitesimal symmetry
of $d\omega$ and $L_E\omega$ and we get that
the above identity is equivalent to
$$
d\omega(L_X\alpha_{1}\Sha,\alpha_{2}\Sha)
+ d\omega(\alpha_{1}\Sha,L_X\alpha_{2}\Sha)
=
d\omega((L_X\alpha_{1})\Sha,\alpha_{2}\Sha)
+ d\omega(\alpha_{1}\Sha,(L_X\alpha_{2})\Sha)\,.
$$ 
By Lemma \ref{Lm: 2.15} $ L_X\alpha_{i}\Sha = (L_X\alpha_{i})\Sha $ which proves Theorem \ref{Th: 2.19}.
\end{proof}

\begin{rmk}\label{Rm: 2.5}
{\rm
We have
\begin{align}\label{Eq: 2.22}
\db[(\alpha_1 &,h_1);
(\alpha_2,h_2)\db] = \frac 12 \big( L_{X_{(\alpha_1,h_1)}}(\alpha_2,h_2)
- L_{X_{(\alpha_2,h_2)}}(\alpha_1,h_1)
\big)\,. 
\end{align}
Really
\begin{align*}
 L_{X_{(\alpha_1,h_1)}}(\alpha_2,h_2)
& - L_{X_{(\alpha_2,h_2)}}(\alpha_1,h_1) =
\\
& = 
\big(
2\, d\Lambda(\alpha_1,\alpha_2) ;
\alpha_1\Sha{\tecka}h_2 
- \alpha_2\Sha{\tecka}h_1
+ h_1\, E{\tecka}h_2
 - h_2\, E{\tecka}h_1
\big)
\end{align*}
and from (2) and (3) of Lemma \ref{Lm: 2.11} we have
\begin{align*}
\alpha_1\Sha{\tecka}h_2 
- \alpha_2\Sha{\tecka}h_1
& =
2\, d\omega(\alpha_1\Sha, \alpha_2\Sha)  + h_1 \, d\omega(E,\alpha_1\Sha) - h_2\, d\omega(E,\alpha_1\Sha)
\\
& = 
2\, d\omega(\alpha_1\Sha, \alpha_2\Sha)  + h_1 \, dE\tecka h_2 - h_2\, E\tecka h_1
\end{align*}
which implies \eqref{Eq: 2.22}.

}
\end{rmk}



\end{document}